\documentclass[11pt,reqno]{amsart}
 \usepackage{amsgen, amstext,amsbsy,amsopn, amsthm, amsfonts,amssymb,amscd,amsmath,euscript,enumerate,url,verbatim,calc,xypic,tikz}
\usepackage{hyperref}
\usepackage{MnSymbol}
\usepackage{pgfkeys}
\usepackage{mathtools}
\usepackage{eqnarray,amsmath}

\def\multiset#1#2{\ensuremath{\left(\kern-.3em\left(\genfrac{}{}{0pt}{}{#1}{#2}\right)\kern-.3em\right)}}

\usetikzlibrary{arrows}
\newcommand{\midarrow}{\tikz \draw[-triangle 90] (0,0) -- +(.1,0);}
 \usepackage{latexsym}
 \usepackage{graphics}
 \usepackage{color}
\usepackage{lastpage}
\usepackage{fancyhdr}
\usepackage{multirow}
\allowdisplaybreaks
\usepackage{graphicx}
\graphicspath{ {F:/IMAGES/} }

 \newcommand{\m}{\mathfrak{m} }

  \newcommand{\Ass}{\operatorname{Ass}}

\newcommand{\rad}{\operatorname{rad}}
  \newcommand{\reg}{\operatorname{reg}}

\newcommand{\proset}{\,\mathrel{\lower 4pt\hbox{$\scriptscriptstyle/$}
\mkern -14mu\subseteq }\,} 

 \newtheorem{theorem}{Theorem}[section]
 \newtheorem{corollary}[theorem]{Corollary}
 \newtheorem{lemma}[theorem]{Lemma}
 \newtheorem{proposition}[theorem]{Proposition}

\usepackage{amsmath}
\newtheorem{notation}[theorem]{Notation}

 \theoremstyle{definition}
 
 \newtheorem{remark}[theorem]{Remark}
 \newtheorem{definition}[theorem]{Definition}

\title[Symbolic powers in weighted oriented graphs] {Symbolic powers  in weighted oriented graphs}
\author[ M. Mandal and D.K. Pradhan ]{Mousumi Mandal$^*$ and Dipak Kumar Pradhan }

\thanks{AMS Classification 2010: 05C22, 13D02, 13D45, 05C25, 05C38, 05E40}
\address{Department of Mathematics, Indian Institute of Technology Kharagpur, 721302, India} \email{mousumi@maths.iitkgp.ac.in}
\address{Department of Mathematics, Indian Institute of Technology Kharagpur, 721302, India}\email{dipakkumar@iitkgp.ac.in}

\begin{document}
\maketitle

\begin{abstract}
Let $D$ be a weighted oriented graph with the underlying graph $G$ when vertices with non-trivial weights  are sinks and $I(D), I(G) $ be the edge ideals corresponding to $D$ and $G,$ respectively. We give an explicit description of the symbolic powers of $I(D)$ using the concept of strong vertex covers.  We show that the ordinary and symbolic powers of $I(D)$ and $I(G)$ behave in a similar way. We provide a description for symbolic powers and Waldschmidt constant of $I(D)$ for certain classes of weighted oriented graphs. When $D$ is a weighted oriented odd cycle,  we compute $\reg (I(D)^{(s)}/I(D)^s)$ and prove $\reg I(D)^{(s)}\leq\reg I(D)^s$ and show that equality holds when there is only one  vertex with non-trivial weight.\\\\  

\end{abstract}

\section{Introduction}

A \textit{directed graph} or\textit{ digraph} $D=(V(D),E(D))$ consists of a finite set $V(D)$ of vertices together with a prescribed collection $E(D)$ of ordered pairs of distinct vertices called edges or arrows. If $ (u,v) \in E(D),$ then we call it  a  directed edge where the direction is from $ u $ to $ v $ and $ u $ (respectively $ v $) is called the initial vertex (respectively the terminal vertex). An oriented graph is a directed graph without multiple edges or loops. In other words, 
an oriented graph $D$ is a simple graph $G$ together with an orientation of its edges. 
An oriented graph $D$ is called vertex weighted oriented  if it is equipped with a weight function $w:V(D)\longrightarrow \mathbb N$. In fact, a vertex weighted oriented graph $D$ is a triplet $D=(V(D), E(D), w)$ where  $V(D)$ and  $E(D)$ are the vertex set and edge set,  respectively, and the weight of a vertex $x_i\in V(D)$ is $w(x_i)$  denoted by $w_i$ or $ w_{x_i}.$   We set $V^+(D)= \{x\in V(D)~ |~ w(x) \neq 1\}  $ and it is denoted by $V^+$.  If $V(D)=\{x_1,\ldots ,x_n\},$ we can regard each vertex $x_i$ as a variable and consider the polynomial ring $R=k[x_1,\ldots ,x_n]$ over a field $k$. 
 Then the edge ideal of $D$  is defined  as
$$I(D)=(x_ix_j^{w_j}~|~(x_i,x_j)\in E(D)) .$$ The underlying graph of $ D $ is the simple graph $ G $ whose
vertex set is $V(G) = V(D) $ and whose edge set is $E(G)= \{ \{u, v \}~|~(u, v) \in E(D)  \}   .$ The edge ideal of $ G $ is 
$I(G) = ( uv ~|~  \{u, v\} \in E(G) )  \subseteq R = k[x_1, \ldots , x_n] 	.$
If a vertex $x_i$ of $D$ is a source (i.e., has only arrows leaving $x_i$), we shall always assume $w_i = 1$ because in this case, the definition of $I(D)$ does not depend on the weight of
$x_i.$ If $w(x)=1$ for all $x\in V,$ then $I(D)$ recovers the usual edge ideal of the underlying graph $G$. Here, we shall always mean that a weighted oriented graph is the same as a vertex weighted oriented graph.  
The interest in edge ideals of weighted digraphs comes from coding theory, especially, in the study of Reed-Muller types codes. The edge ideal of a vertex weighted digraph appears as initial ideals of vanishing ideals of certain projective spaces over finite fields \cite{bernal}.
\vspace*{0.2cm}\\
Algebraic invariants and properties like Cohen-Macaulayness and unmixedness of edge ideals of weighted oriented graphs have been 
studied in \cite{H.T-2018} and \cite{pitones}. Recently in \cite{beyarsan} and \cite{zhu}, regularity and projective dimension of the edge ideals of some classes of weighted oriented graphs have been studied. In \cite{pitones}, Y. Pitones et al. studied some properties  of weighted oriented graphs when  $ V(D) $ is a strong vertex cover and the elements
 of   $ V^+ $ are sinks. They described the irreducible decomposition of $ I(D) $ for any weighted oriented graph $ D $ using the concept of strong vertex cover. In the geometrical context,  symbolic powers are important since they capture all the polynomials that vanish with a given multiplicity. For any homogeneous ideal $  I \subset R ,$   its $ m$-th symbolic power is defined   as $I^{(m)}=\displaystyle{\bigcap_{p\in \Ass I}(I^mR_p\cap R)}$. In \cite{zariski}, a classical result of Zariski
 and Samuel states that $ I^{(m)} = I^{m} $
 for all $ m \geq  1 $ if $ I $ is generated by a regular
 sequence, or equivalently, a complete intersection. Ideals that have the property $ I^{(m)} = I^{m} $
 for all $ m \geq  1 $ are called normally torsion free because their Rees
 algebra is normal. In particular, in \cite{gitler} Gitler et al. showed
 that a squarefree monomial ideal is normally torsion free if and only if the
 corresponding hypergraph satisfies the max-flow min-cut property. Symbolic powers of codimension
$  2 $ Cohen-Macaulay ideals have been studied recently in \cite{nagel}. 
 In the last decade the comparison
 between symbolic and regular powers of ideals is studied not only for
 graphs but also for ideal defining a set of points both in projective and
 multiprojective spaces \cite{bocci2016,guardo,harbourne}.  Recently symbolic powers and   invariants of the edge ideals of simple graphs  have been studied in \cite{bidwan,Y.GU,janssen}. But nothing is known about the symbolic powers of edge ideals of weighted oriented graphs.  
 \vspace*{0.2cm}\\    
In general, computing the symbolic powers of ideals is a difficult problem. Symbolic powers of square free monomial ideals have been studied via the concept of minimal vertex cover for edge ideals of graphs. Here we study the symbolic powers of some classes of non square free monomial ideals through the edge ideals  of weighted oriented graphs using the concept of strong vertex cover. We also compare the regularity of the symbolic  and ordinary powers of the edge ideals of weighted oriented odd cycles when the elements of $ V^+ $ are sinks.\\  
 In Section 2, we recall all the definitions and results which will be required for the rest of the paper. In Section 3, we provide a  relation between ordinary  and symbolic powers of edge ideals of  some  classes of weighted naturally oriented graphs in   Corollary \ref{unicycle}. In Theorem \ref{sym.lemma.2}, we show that the $ m$-th symbolic power of edge ideal of a weighted oriented graph, when the elements of $ V^+ $ are sinks, is the intersection of the  $m$-th powers of irreducible ideals associated to the strong vertex covers. In  Theorem \ref{sym.theorem.1}, we prove that the ordinary  and symbolic powers of edge ideal of a weighted  oriented graph behaves in the similar way as in the edge ideal of its underlying graph in case of elements of $ V^+ $ are sinks.    
 \vspace*{0.2cm}\\
 We also give an explicit description of symbolic powers and Waldschmidt constants of edge ideals of weighted oriented graphs when the elements of $ V^+ $ are sinks and underlying graphs are unicyclic graph with a unique odd cycle, complete graph and clique sum of two odd cycles of same length joining at a common vertex  (see Proposition \ref{uni-sym}, Proposition \ref{com-sym} and Proposition   \ref{clique-sym}, respectively). In Section 4, in Theorem \ref{maximal.2}, we provide an inequality between the regularity of ordinary  and symbolic powers of  edge ideal  of a weighted oriented odd cycle  when the elements of $ V^+ $ are sinks  and in Corollary \ref{maximal.5}, we prove the equality when  $ V^+ $ contains a single element.            
\section{Preliminaries}    
In this section,  we recall some definitions and results regarding a weighted oriented graph $ D   $ and its underlying graph $ G .$ 
A vertex cover of $ G $ is a subset $ V^{\prime} \subseteq V(G)  $ such that for all $e \in E(G), $  $ e \cap V^{\prime} \neq \phi.$ A minimal vertex
cover is a vertex cover which is minimal with respect to inclusion.\\
The next two lemmas describe the symbolic powers of the edge ideals in terms of minimal vertex covers.
\begin{lemma}\cite[Lemma 2.5]{bocci2016}\label{o.s.1}
	Let $G$ be a graph on vertices $\{x_1,\ldots,x_n\},$  $ I=I(G)\subseteq k[x_1,\ldots,x_n]$ be the edge ideal of $G$ and $V_1,\ldots,V_r$ be the minimal vertex covers of $G.$ Let $P_j$ be the monomial prime ideal  generated by the variables in $V_j.$ Then $$I=P_1 \cap\cdots\cap P_r$$ and $$I^{(m)}=P_1^{m} \cap\cdots\cap P_r^{m}.$$
\end{lemma}

\begin{lemma}\cite[Lemma 2.6]{bocci2016}\label{o.s.2}
	Let $I \subseteq S$ be a squarefree monomial ideal with minimal primary decomposition $I=P_1\cap\cdots \cap P_r ~~with ~~P_j = (x_{j_1},\ldots ,x_{js_j})$ for $j= 1,\ldots ,r$. Then $ {x_1^{a_1}}\cdots {x_n^{a_n}} \in I^{(m)} \mbox{ if and only if } a_{j_1}+\cdots +a_{j_{s_j}}\geq m$ for $j=1,\ldots,r$.
	
\end{lemma}
Below we recall some definitions and results for the weighted oriented graph $ D .$ 
\begin{definition}
	A vertex cover $ C $ of $ D $ is a subset of $  V(D)  $ such that if $ (x, y) \in E(D) ,$ then
	$ x \in C $ or  $ y \in C . $ A vertex cover $ C $ of $ D $ is minimal if each proper subset of $ C $ is not a
	vertex cover of $ D. $ We set $(C)$ to be the ideal generated by the variables in $C.$
\end{definition}
\begin{remark}\cite[Remark 2]{pitones}\label{s.v.00}
	$ C $ is a minimal vertex cover of $ D $ if and only if $ C $ is a minimal vertex cover of $ G. $
\end{remark}
\begin{definition}
	Let $ x $ be a vertex of a weighted oriented graph $ D, $ then the sets $ N_D^+ (x) =	\{y : (x, y) \in E(D)\}  $ and  $ N_D^- (x) = \{y : (y, x) \in E(D)\}  $ are called the out-neighbourhood and the in-neighbourhood of $ x $, respectively. Moreover, the neighbourhood of $ x $ is the set $ N_D(x) = N_D^+ (x)\cup N_D^- (x) .$  Define $\deg_D(x) = |N_D(x)|$ for  $ x \in V(D) $.  A vertex $ x \in V(D) $ is called a source vertex if $N_D (x)= N_D^+ (x) .$ 
A vertex $ x \in V(D) $ is called a sink vertex if $N_D (x)= N_D^- (x) .$  
\end{definition}
\begin{definition}\cite[Definition 4]{pitones}
Let $ C $ be a vertex cover of a weighted oriented graph $ D, $ we define \vspace*{0.2cm}\\
\hspace*{3cm}$ L_1(C) = \{x \in C ~|~ N_D^+ (x) \cap C^c \neq        \phi \}, $ \vspace*{0.2cm}\\
\hspace*{2.85cm} $L_2(C) = \{x \in C ~|~x\notin L_1(C) ~\mbox{and}~  N_D^- (x) \cap C^c \neq   \phi \}$ and  \vspace*{0.2cm}\\ 
\hspace*{2.8cm}  $ L_3(C) = C \setminus (L_1(C) \cup L_2(C))$ \vspace*{0.2cm}\\
 	where $ C^c $ is the complement of $  C ,$ i.e., $ C^c = V(D) \setminus C. $      	
\end{definition}
\begin{lemma}\cite[Proposition 6]{pitones}\label{s.v.0}
Let $ C $ be  a  vertex  cover  of $ D. $  Then $ L_3(C) =\phi $ if  and  only  if $ C $ is  a minimal vertex cover of $ D .$  	

\end{lemma}
 
\begin{definition}\cite[Definition 7]{pitones}
A vertex cover $ C $ of $ D $ is strong if for each $ x \in L_3(C)  $ there is $ (y, x) \in 
E(D) $ such that $ y \in L_2(C) \cup L_3(C)$ with $  y \in V^+$   (i.e., $ w(y) \neq  1 $).   
\end{definition}

\begin{definition}\cite[Definition 32]{pitones}
	A weighted oriented graph $ D $ has the minimal-strong property if each
	strong vertex cover is a minimal vertex cover.

\end{definition}
The following lemma gives a class of weighted oriented graphs which satisfy the minimal-strong property.
\begin{lemma}\cite[Lemma 47]{pitones}\label{s.v.4}
	If the elements of $ V^+ $ are sinks, then $ D $ has the minimal-strong property. 	
\end{lemma}
The next lemma describes the irreducible decomposition of the edge ideal of  weighted oriented graph $ D $ in terms of irreducible ideals associated with the strong vertex covers of $D$.

\begin{lemma}\cite[Theorem 25, Remark 26]{pitones}\label{s.v.2}
Let $ D $ be  a weighted oriented graph and $C_1,\ldots, C_s$ are the strong vertex covers of $ D ,$ then  the irredundant irreducible decomposition of $ I(D) $ is
    $$I(D) =  I_{C_1} \cap\cdots\cap I_{C_s} $$ 
where each $ I_{C_i} = ( L_1(C_i) \cup \{x_j^{w(x_j)}~|~x_j \in L_2(C_i) \cup L_3(C_i)\} ) ,$   $ \rad(I_{C_i})=P_i = (C_i)$.
\end{lemma}
\begin{corollary}\cite[Remark 26]{pitones}\label{s.v.3}
	Let $ D $ be a weighted oriented graph. Then $ P $ is an associated 
	prime of $ I(D) $ if and only if $ P = (C) $ for some strong vertex cover $ C $ of $ D. $
\end{corollary}
\begin{notation}
The	degree of a monomial $m \in  k[x_1,\ldots,x_n]$ is denoted by $\deg(m).$

\end{notation}

Now we recall the definition of Castelnuovo-Mumford
 regularity using local cohomology. Let $M$ be a non zero  module over the polynomial ring $R$.
 If $  H^i_\m(M) $
 denotes the $ i$-th  local cohomology module of the  $ R$-module $ M $
 with support on the  maximal ideal $ \m ,$ we set $$ a_i(M) = \max\{   j \in \mathbb{Z} : [H^i_\m(M)]_j \neq 0 \}  $$ (or $a_i(M)  =  -\infty $ if $ H^i_\m(M) =   0 $) where $ [H^i_\m(M)]_j $ denotes the $ j $-th graded component of the $ i$-th local cohomology
 module of $ M. $ Then the Castelnuovo-Mumford
 regularity   of $ M $ is defined as  $$ \reg M = \max_{0 \leq i \leq \dim M} \{a_i(M) + i\}. $$
 One of the important asymptotic invariants for symbolic powers is resgurence which was introduced by Bocci and Harbourne \cite{bocci2010} to answer the  ideal containment problem. 
In general, computation of resurgence of an ideal is difficult. In \cite{bocci2010}, one lower bound of resurgence of an ideal is given in terms of  another invariant known as Waldschmidt constant of $ I. $   Let $\alpha(I)$ $ = \min\{d ~|~ I_d \neq 0\}, $ i.e., $\alpha(I)$  is the smallest degree of a nonzero element in $ I.$ Then the  Waldschmidt constant  of $ I $ is  defined as $$\widehat{\alpha}{(I)} = \displaystyle{\lim_{s\rightarrow\infty}{} \frac{\alpha{(I^{(s)})}}{s}}.$$

\section{Symbolic Powers in Oriented Graphs}
We provide an explicit description of the symbolic powers for edge ideals of weighted oriented graphs when the elements of $ V^{+}  $ are sinks. In the next lemma, we give a sufficient condition for the equality of the symbolic and ordinary powers of the edge ideals for some classes of weighted oriented graphs.  

\begin{lemma}\label{sym.lemma.1}
Let $I = I(D)$ be the edge ideal of a weighted oriented graph $ D .$ If $ V(D) $ is a strong vertex cover of $ D ,$ then $ I^{(k)}=I^k $ for all $ k. $	
\end{lemma}
\begin{proof}
Since $ V(D) $ is the strong vertex cover of $ D ,$ by Corollary \ref{s.v.3}, $ P = (V(D)) $ is an associated prime of $ I(D) .$ Here $ P $ is the unique maximal associated prime in $\Ass(I).$ Thus  by \cite[Lemma 3.3]{cooper}, we have $ I^{(k)}=I^k $ for all $ k. $  
\end{proof}
 A  cycle is naturally oriented if  all edges of cycle are oriented in clockwise direction. In a naturally oriented unicyclic graph, the cycle is naturally oriented and each edge of the tree  connected with the cycle is oriented away from the cycle. 
 \begin{corollary}\label{unicycle}
 	Let $I = I(D)$ be the edge ideal of a weighted  oriented  graph $D.$  If there are weighted naturally oriented unicyclic graphs   $ D_1,\ldots,D_s $ of $ D $ such that $ V(D_1),\ldots,V(D_s) $ is a partition of   $ V(D)$ and   $ w(x) \neq 1 $ if $ \deg_{D_i}(x) > 1 $ for each $ i ,$
 	then $ I^{(k)}=I^k $ for all $ k. $
 \end{corollary}
 \begin{proof}
 	 Since there are weighted naturally oriented unicyclic graphs   $ D_1,\ldots,D_s $ of $ D $ such that $ V(D_1), \ldots,V(D_s) $ is a partition of $ V(D)$ and  $ w(x) \neq 1 $ if $ \deg_{D_i}(x) > 1 $ for each $ i ,$   by \cite[Proposition 15]{pitones},  $ V(D) $ is a strong vertex cover of $ D. $ Therefore by Lemma    \ref{sym.lemma.1}, we have $ I^{(k)}=I^k $ for all $ k. $ 	
\end{proof}

\begin{corollary}\label{tournament}
	Let $I = I(D)$ be the edge ideal of a weighted oriented complete graph $D$ on the vertex set $ \{x_1,\ldots,x_n\} $ where the cycle $ C_n=(x_1,\ldots,x_n) $ is naturally oriented and the    diagonals are oriented in any direction such that $ w(x) \neq 1 $ for any vertex $x.$ Then $ I^{(k)}=I^k $ for all $ k. $ 
\end{corollary}

\begin{proof}
Let $ D_1 $ be the weighted naturally oriented cycle whose underlying graph is  the cycle $ C_n=(x_1,\ldots,x_n) .$  Here $ V(D_1) = V(D) $ and  $ w(x) \neq 1 $ for all  $ x \in V(D_1) .$   Then by Corollary \ref{unicycle}, we have  $ I^{(k)}=I^k $ for all $ k. $
\end{proof}
In the next theorem, we describe  the symbolic powers for edge ideals of  weighted oriented graphs when
      the elements of $ V^{+}  $ are sinks, in terms of irreducible ideals associated to the strong vertex covers. For the remainder of this section,  we set $ V^+ \neq \phi. $
\begin{theorem}\label{sym.lemma.2}
Let $I = I(D)$ be the edge ideal of a weighted oriented graph $ D $ where  the elements of $ V^{+}  $ are sinks and $ C_1,\ldots,C_s $ be the strong vertex covers of $ D .$  Let $ w_j = w(x_j) $ if $ x_j \in V^+. $ Then  the irredundant irreducible decomposition of $ I $ is
$$ I= I_{C_1}\cap \cdots \cap I_{C_s}$$  and $$ I^{(m)}   = {I_{C_1}}^{m}\cap \cdots \cap {I_{C_s}}^{m} $$
  where each $ I_{C_i} = ( \{ x_j~|~x_j \in C_i \setminus V^{+}\} \cup \{ x_j^{w_j}~|~x_j \in C_i \cap V^{+} \})  ,$   $  \rad(I_{C_i}) = P_i = (C_i)$    and $ P_i $'s are minimal primes of $ I. $  	
\end{theorem}
\begin{proof}
Here the elements of $ V^{+}  $ are sinks.	
By Lemma \ref{s.v.2}, $ I= I_{C_1}\cap \cdots \cap I_{C_s}$ be the irredundant irreducible decomposition of $ I $ where each $ I_{C_i} = ( L_1(C_i) \cup \{x_j^{w_j}~|~x_j \in L_2(C_i) \cup L_3(C_i)\} ) $, $  \rad(I_{C_i}) = P_i = (C_i)$    and $ P_i $'s are associated primes of $ I. $	By Lemma \ref{s.v.4}, $C_1,\ldots,C_s $ are minimal  vertex covers of $ D $. Thus    by Lemma \ref{s.v.0}, we have $L_3(C_i) = \phi $  for $ i=1,\ldots,s .$   For a fixed $ i ,$ if $ x\in C_i $ is a sink vertex, then  $ N_{D}^{+}(x) \cap  C_i^c = \phi, $ which implies $ x  \notin L_1(C_i)  .$ Thus $ L_2(C_i)= C_i \cap V^{+} $ and   $ L_1(C_i)= C_i \setminus V^{+} .$ Therefore $ I_{C_i} =  ( \{ x_j~|~x_j \in C_i \setminus V^{+}\} \cup \{ x_j^{w_j}~|~x_j \in C_i \cap V^{+} \})  .$ Since  $ C_i $'s are minimal vertex covers of $ D $,  the $ P_i$'s are minimal primes of   $ I .$ Hence by \cite[Lemma 3.7]{cooper}, we have  $ I^{(m)}   = {I_{C_1}}^{m}\cap \cdots \cap {I_{C_s}}^{m}. $   
\end{proof}

\begin{proposition}\label{sym.cor.1}
	Let $ I $ be the edge ideal of a weighted oriented graph $ D $ where  the elements of $ V^{+}  $ are sinks with irredundant irreducible decomposition  $ I= I_{C_1}\cap \cdots \cap I_{C_s}$ where  $  \rad(I_{C_i}) = P_i = (C_i) = (x_{i_1},\ldots,x_{i_{r_i}})$ for $ i=1,\ldots,s $ and $ C_i $'s are the strong vertex covers of $D.$ Let $ w_j = w(x_j) $ if $ x_j \in V^+. $ Then  every minimal  generator of $ I^{(m)}$ is of the form $ (\prod x_j^{a_j}| x_j \notin V^{+})( \prod {x_j^{w_j}}^{a_j}| x_j \in V^{+} ) $ for some $ a_j $'s   where  $ a_{i_1}+\cdots+a_{i_{r_i}} \geq m $ for $ i=1,\ldots,s. $

\end{proposition}

\begin{proof}
Here each $I_{C_i}  $ is of the form $ I_{C_i} = ( \{ x_j~|~x_j \in C_i \setminus V^{+}\} \cup \{ x_j^{w_j}~|~x_j \in C_i \cap V^{+} \}) $ and  $  \rad(I_{C_i}) = P_i = (C_i) = (x_{i_1},\ldots,x_{i_{r_i}}).$  By Theorem \ref{sym.lemma.2}, we have $ I^{(m)}   = {I_{C_1}}^{m}\cap \cdots \cap {I_{C_s}}^{m}. $ Let $ t =  (\prod x_j^{a_j}| x_j \notin V^{+})( \prod {x_j^{w_j}}^{a_j}| x_j \in V^{+} )(\prod x_j^{b_j}| x_j \in V^{+},b_j < w_j)$  for some $ a_j $'s and $ b_j $'s. Then $ t \in I^{(m)} $ if and only if $ t \in {I_{{C_i}}}^m $ for $ i=1,\ldots,s. $ This is possible if and only if there exists at least one generator $ f_i \in {I_{C_i}}^m $ such that $ f_i $ divides $ t $ for $ i=1,\ldots,s  $ which is equivalent to the condition $ a_{i_1}+\cdots+a_{i_{r_i}} \geq m $ for $ i=1,\ldots,s. $\\
From the structure of  $ I^{(m)}   = {I_{C_1}}^{m}\cap \cdots \cap {I_{C_s}}^{m},$ it is clear that  every minimal  generator of $ I^{(m)}$ is of the form $ (\prod x_j^{a_j}| x_j \notin V^{+})( \prod {x_j^{w_j}}^{a_j}| x_j \in V^{+} ) $ for some $ a_j $'s   where  $ a_{i_1}+\cdots+a_{i_{r_i}} \geq m $ for $ i=1,\ldots,s. $    
\end{proof}

Now we relate the  symbolic powers of  edge ideals of weighted oriented graphs with the edge ideals of underlying graphs and show that both of them behave in a similar way when the  elements of $V^+$ are sinks.
\begin{notation} \label{phi}    
	   Let $ D $ be a  weighted oriented graph  where  the elements of $ V^{+}  $ are sinks and $ w_j =w(x_j) $  if $ x_j \in V^{+} .$ Here $R=k[x_1,\ldots,x_n]=$  $\displaystyle{{\bigoplus_{d=0}^{\infty}}R_d}$ is the standard graded polynomial ring. Consider the map
\begin{align*}
 \Phi : R \longrightarrow R ~ \mbox{where}  ~ x_j \longrightarrow x_j ~ \mbox{if} ~ x_j \notin V^{+}  \mbox{and} ~ x_j \longrightarrow x_j^{w_j}~  \mbox{if}  ~ x_j \in V^{+} .
\end{align*}
\end{notation}
 Here $ \Phi $ is an injective homomorphism of k-algebras.
 We use this definition of $ \Phi $ for the following theorem.
 \begin{theorem}\label{sym.theorem.1}
 Let $ D $ be a  weighted oriented graph  where   the elements of   $ V^{+}  $ are sinks and  $ G $ be the underlying graph of  $ D.$   Let $I$ and $\tilde{I}$ be the edge ideals of $G $ and  $D ,$ respectively. Then $ \Phi(I^k)  = \tilde{I}^k   $ and $  \Phi(I^{(k)})  = \tilde{I}^{(k)}$ for all $ k \geq 1. $ 	
 \end{theorem}        

\begin{proof}
Let $ V(G) = V(D) = \{x_1,\ldots,x_n\}. $
Let $ C_1,\ldots,C_s $ are the strong vertex covers of $ D $ and by  Lemma \ref{s.v.4}, they are   minimal vertex covers of $ D $. By Remark \ref{s.v.00}, $ C_1,\ldots,C_s $ are also  minimal vertex covers of $ G. $ Thus by Lemma \ref{o.s.1}, the minimal primary decomposition of $ I $ is $ I = P_1 \cap \cdots \cap P_s $ where $ P_i = (C_i) =  (x_{i_1},\ldots,x_{i_{r_i}}) $ for $ i=1,\ldots,s .$ 	By Lemma \ref{s.v.2}, $\tilde{ I}= {\tilde{I}}_{C_1}\cap \cdots \cap {\tilde{I}}_{C_s}$ be the irredundant irreducible decomposition of $ I $  where each $ {\tilde{I}}_{C_i} = ( \{ x_j~|~x_j \in C_i \setminus V^{+}\} \cup \{ x_j^{w_j}~|~x_j \in C_i \cap V^{+} \}) ,$ $ w_j = w(x_j) $ if $ x_j \in V^+ ,$   $  \rad({\tilde{I}}_{C_i}) = P_i = (C_i)$ and $ P_i $'s are minimal primes of $\tilde{I}.$  Here   $\Phi(I)  = {\tilde{I}}. $ Then $ \Phi(I^k)  = [\Phi(I) ]^k= {\tilde{I}^k}    $ for all $ k \geq 1. $\\
 Next we want to prove  that  $ \Phi(I^{(k)})  = {\tilde{I}^{(k)}} .$   By Lemma \ref{o.s.1}, $ I^{(k) } = P_1^{k} \cap \cdots \cap P_s^{k} $ and	by Theorem \ref{sym.lemma.2}, we have $ {\tilde{I}}^{(k)}   = {{\tilde{I}}_{C_1}}^{k}\cap \cdots \cap {{\tilde{I}}_{C_s}}^{k} .$  Here $\Phi(P_i)  = {\tilde{I}}_{C_i} $ for $ i=1,\ldots,s .$  
 \vspace*{0.2cm}\\
Let $ q $ be a minimal generator of $ \Phi(I^{(k)}) .$ Then there exists $ p \in I^{(k)}$ such that $ q = \Phi(p) .$ Since $ p \in I^{(k)} = P_1^{k} \cap \cdots \cap P_s^{k}  ,$ then $p \in P_i^{k} $  for each $ i=1,\ldots,s$  which implies that   $ \Phi(p) \in \Phi({P_i}^k)  = [\Phi(P_i) ]^k= {\tilde{I}_{C_i}^k} .$ Thus $ q = \Phi(p)  \in {{\tilde{I}}_{C_1}}^{k}\cap \cdots \cap {{\tilde{I}}_{C_s}}^{k}  = {\tilde{I}}^{(k)}.$ Therefore $ \Phi(I^{(k)})    \subseteq  {\tilde{I}^{(k)}} .$
\vspace*{0.2cm}\\   
Let $ q $ be  a  minimal generator of $ {\tilde{I}}^{(k)} .$ Then by Proposition \ref{sym.cor.1},  we have $   q=  (\prod x_j^{a_j}| x_j \notin V^{+})( \prod {x_j^{w_j}}^{a_j}| x_j \in V^{+} )  $ for some $ n$-tuple $ (a_1,\ldots,a_n) $  where  $ a_{i_1}+\cdots+a_{i_{r_i}} \geq k $ for $ i=1,\ldots,s. $ By letting  $ p:=x_1^{a_1}\cdots x_n^{a_n},$ we see that  $ \Phi(p) = q  $ and by Lemma    \ref{o.s.2}, $ p \in I^{(k)}$ since   $ a_{i_1}+\cdots+a_{i_{r_i}} \geq k $  for $ i=1,\ldots,s $. Thus $ q=\Phi(p) \in  \Phi(I^{(k)}) .$ Therefore $ {\tilde{I}^{(k)}}   \subseteq  \Phi(I^{(k)})  .$ Hence  $ \Phi(I^{(k)})   =  {\tilde{I}^{(k)}}      $    for all $ k \geq 1. $         
\end{proof}

Next, we apply the above theorem to some specific classes of weighted oriented graphs and give an explicit description of the symbolic powers of their edge ideals. Then we compute the Waldschmidt constant.    
\begin{corollary}
	Let $\tilde{I}$ be the edge ideal of a  weighted oriented bipartite graph $ D $ where the elements of   $~V^{+}$ are sinks and let $ G $ be the underlying graph  of $ D $ which is a bipartite graph. Then 
	$ {\tilde{I}^{(k)}} = {\tilde{I}^{k}} $ for all $ k \geq 1. $ 
	
\end{corollary}    

\begin{proof}
	Let $ I $ be the edge ideal of $ G $ and $ \Phi $ is same as defined in  Notation \ref{phi}. Since $ G $ is a bipartite graph, we have   $ {{I}^{(k)}} = {{I}^{k}} $ for all $k \geq 1$   by \cite[Theorem 5.9]{simis}. 
	Thus by Theorem \ref{sym.theorem.1},  we have  $\Phi( {{I}^{(k)}})  =  \Phi({{I}^{k}}) ,$ i.e., $ {\tilde{I}^{(k)}} = {\tilde{I}^{k}} $ for all $k \geq 1.$  
\end{proof} 
In \cite[Theorem 6.7]{bocci2016}, C. Bocci et al. found that the Waldschmidt constant of edge ideal of an odd cycle $C_{2n+1}$ is $\displaystyle{\frac{2n+1}{n+1}}$.  
In the next Proposition, we observe that the Waldschmidt constant of edge ideal of a weighted oriented unicyclic graph depends on the position of    the  elements of $ V^{+} .$
\begin{proposition}\label{uni-sym}
Let $\tilde{I}$ be the edge ideal of a weighted oriented unicyclic graph $ D $ where the elements of $ V^{+} $ are sinks and let $ G $ be the underlying graph  of $ D $ with a unicycle $ C_{2n+1} = (x_1, \ldots , x_{2n+1} ) . $ Let $ c = (\prod x_j| x_j \in  V(C_{2n+1}) \setminus V^{+})( \prod {x_j^{w(x_j)}}| x_j \in  V(C_{2n+1}) \cap V^{+})   .$         
\begin{enumerate}  
	\item Then  $ {\tilde{I}^{(k)}} = {\tilde{I}^{k}} $ for $  1\leq k \leq n $ 
	 and  $ {\tilde{I}^{(n+1)}} = {\tilde{I}^{n+1}} +  (c). $
	 \item Let $ s \in \mathbb{N} $ and write  $ s = k(n + 1) + r $ for some $ k \in  \mathbb{Z}  $ and
	 $ 0 \leq r \leq n .$ Then   $$ \tilde{I}^{(s)} =   \displaystyle{ \sum_{t=0}^{k}  \tilde{I}^{s-t(n+1)}  (c)^t }  .$$     
	\item If none of the vertices of $ V(C_{2n+1}) $ is in   $ V^{+} $ and some vertex of $V(G) \setminus V(C_{2n+1})$ is  in   $ V^{+} ,$ then  $ \alpha{(\tilde{I}^{(s)})}$
	= $ 2s- \Bigl\lfloor\dfrac{s}{n+1}\Bigr\rfloor\qquad$ for $ s \in \mathbb{N} .$\\
	In particular, the Waldschmidt constant of $\tilde{I}$ is given by \begin{align*}\widehat{\alpha}{(\tilde{I})}=\frac{2n+1}{n+1}.\end{align*}
	
	\item If at least one of the vertices of  $ V(C_{2n+1}) $ is in $ V^{+} $, then  $ \alpha{(\tilde{I}^{(s)})}$
	= $ 2s$ for $ s \in \mathbb{N} .$\\
	In particular, the Waldschmidt constant of $\tilde{I}$ is given by \begin{align*}\widehat{\alpha}{(\tilde{I})}= 2.\end{align*}   
\end{enumerate}
\end{proposition}
\begin{proof}
Let $I$ be the edge ideal of $ G $ and $ \Phi $ is same as defined in  Notation \ref{phi}.
\begin{enumerate}
		\item  Since $ G $ is  a unicyclic graph with a unique odd cycle  $ C_{2n+1} = (x_1, \ldots , x_{2n+1} ) , $  by \cite[Lemma 3.3]{Y.GU}  , we have $  {{I}^{(k)}} = {{I}^{k}} $ for $  1\leq k \leq n $ 
and $ {{I}^{(n+1)}} = {{I}^{n+1}} + (x_1\cdots x_{2n+1}). $ Thus, by Theorem \ref{sym.theorem.1},  $\Phi( {{I}^{(k)}})  =  \Phi({{I}^{k}}),$ i.e., $ {\tilde{I}^{(k)}} = {\tilde{I}^{k}} $  for $  1\leq k \leq n $  and  $ \Phi({{I}^{(n+1)}}) = \Phi({{I}^{n+1}}) + \Phi(x_1\cdots x_{2n+1}) ,$ i.e., $ {\tilde{I}^{(n+1)}} = {\tilde{I}^{n+1}} +   (c)      . $  	

\item  By \cite[Lemma 3.4]{Y.GU}, we have $ {I}^{(s)} =   \displaystyle{ \sum_{t=0}^{k}  {I}^{s-t(n+1)}  (x_1 \cdots  x_{2n+1})^t  .}$ Then by Theorem \ref{sym.theorem.1}, applying $ \Phi $ on both sides, we get   $ \tilde{I}^{(s)} =   \displaystyle{ \sum_{t=0}^{k}  \tilde{I}^{s-t(n+1)}  (c)^t  .}$ 
\item Assume that none of the vertices of $ V(C_{2n+1}) $ is in $ V^{+} $. Thus by (1),  $ {\tilde{I}^{(n+1)}} = {\tilde{I}^{n+1}} +  (c) =  {\tilde{I}^{n+1}} +     (x_1\cdots x_{2n+1}). $ Since all the  edges of cycle are without any element of $ V^{+} $,    $ \alpha(\tilde{I}^{s})=\alpha (I^{s}) = 2s $ for all $ s. $ Here $ \alpha(c) =2n+1  ,$ thus we have   $ \alpha(\tilde{I}^{(n+1)})=\alpha (I^{(n+1)}) = 2n + 1 .$ Since ${I}^{(s)}  $ is generated by $ {I}  $ and $ {I}^{(n+1)}, $  $\tilde{I}^{(s)}  $ is generated by $ \tilde{I}  $ and $\tilde{I}^{(n+1)} $ which implies   $ \alpha{(\tilde{I}^{(s)})}= \alpha{({I}^{(s)})}. $ Therefore by \cite[Theorem 3.6]{Y.GU}, we have $ \alpha{(\tilde{I}^{(s)})}   $ 
$\displaystyle{ = 2s-\Bigl\lfloor\dfrac{s}{n+1}\Bigr\rfloor }. $   Hence $\displaystyle{\widehat{\alpha}{(\tilde{I})}= 2 - \frac{1}{n+1}=\frac{2n+1}{n+1}}.$    
\item   Assume that at least one of the vertices of $ V(C_{2n+1}) $ is in $ V^{+} $. By (1), we have $ \tilde{I}^{(s)} =  \tilde{I}^s + \displaystyle{ \sum_{t=1}^{k} \tilde{I}^{s-t(n+1)}  (c)^t  .}$ Since $ G $ contains a unique odd cycle, there is an edge   of odd cycle without any  element of $ V^{+} .$ Thus  $ \alpha(\tilde{I}^s)=\alpha (I^s)= 2s $ for all $ s. $ Since the cycle $ C_{2n+1} $ contains at least one  element of $ V^{+} ,$ $   \alpha(c) \geq 2(n+1) $ and $   \alpha(c^t)\geq 2t(n+1) .$ Also $ \alpha (\tilde{I}^{s-t(n+1)}) =   2s - 2t(n+1) $ and $ \alpha ({\tilde{I}}^{s-t(n+1)} (c)^t) \geq 2s - 2t(n+1) + 2t (n+1) = 2s  $ for $ 1 \leq t \leq k. $ Since $ \alpha(\tilde{I}^s) = 2s,   $   $\alpha(\tilde{I}^{(s)}) = \alpha(\tilde{I}^s) = 2s.   $   Hence $\widehat{\alpha}{(\tilde{I})}= 2.$            	  
\end{enumerate}
\end{proof}

\begin{proposition}\label{com-sym}
	Let $ \tilde{I} $ be the edge ideal of a weighted oriented complete graph $ D $  where the elements of $ V^+ $ are sinks and let $ G $ be the underlying  graph  of $ D $  on the vertices $\{x_1,\ldots,x_{n}\}.$ Then
	\begin{enumerate}
		
		\item For any $ s  \geq   n ,$ we have  $\displaystyle{\tilde{ I}^{(s)}=\sum_{\substack{(r_1,\dots,r_{n-1})~and~\\ s=r_1+2r_2+\dots +(n-1)r_{n-1}}}
			{\tilde{I}^{r_1} {{\tilde{I}}^{ {(2)} ^{r_2} } } \cdots {{\tilde{I}}^{ {(n-1)} ^ {r_{n-1}} } }  }}.$
		\item For any $ s\in \mathbb{N}$, $ \alpha{(\tilde{I}^{(s)})}$
		$\displaystyle{ = s+ \Bigl\lceil\dfrac{s}{n-2}\Bigr\rceil.\qquad}$\\
			In particular, the Waldschmidt constant of $\tilde{I}$ is given by \begin{align*}\widehat{\alpha}{(\tilde{I})} =\frac{n-1}{n-2}.\end{align*}
		
	\end{enumerate}
\end{proposition}

\begin{proof}
  Let $ I $ be the edge ideal of $G $  and $ \Phi $ is same as defined in  Notation \ref{phi}.
\begin{enumerate}

	  	\item By \cite[Theorem 4.2]{bidwan} for any $s\geq n,$ we have
		$$I^{(s)}=\sum_{\substack{(r_1,\dots,r_{n-1})~and~\\ s=r_1+2r_2+\dots+ (n-1)r_{n-1}}}{I^{r_1}{{I^{(2)}}^{r_2}}\cdots{{I^{(n-1)}}^{r_{n-1}}}}.$$ By Theorem \ref{sym.theorem.1}, applying $ \Phi $ on both sides we get $$\displaystyle{\tilde{ I}^{(s)}=\sum_{\substack{(r_1,\dots,r_{n-1})~and~\\ s=r_1+2r_2+\cdots +(n-1)r_{n-1}}}
			{\tilde{I}^{r_1} {{\tilde{I}}^{ {(2)} ^ {r_2} } } \cdots {{\tilde{I}}^{ {(n-1)} ^ {r_{n-1}} } }  }}.$$ 
		\item Since $ D $   is a weighted oriented complete graph   and all the elements of $ V^+ $ are sinks, there is only one  element in $ V^{+} .$ Without loss of generality, we may assume that the only element in  $ V^+ $ is $ x_n $ where  $ w_n:=w(x_n).$ For $1\leq s\leq n-1$ by the definition of $I^{(s)},$ no monomial of degree $s$ is in $I^{(s)}.$
		 Note that  for  $1\leq s\leq n-1 , $  $x_1x_2\cdots x_{s+1}\in I^{(s)}.$ Then  by Theorem \ref{sym.theorem.1},  $\Phi(x_1x_2\cdots x_{s+1}) = x_1x_2\cdots x_{s+1}\in {\tilde{ I}}^{(s)}$  for $1\leq s\leq n-2$ and  $\Phi(x_1x_2\cdots x_{n}) =   x_1x_2\cdots {x_{n}}^{w_n}\in {\tilde{ I}}^{(n-1)}.$ Since degree of any monomial can not be reduced under the map $\Phi,$  $ x_1x_2\cdots x_{s+1}$ is one of the least degree generator in both $I^{(s)}$ and  $ {\tilde{ I}}^{(s)}$ for $1\leq s\leq n-2.$ Hence   $\alpha({\tilde{ I}}^{(s)})=\alpha({{ I}}^{(s)})=s+1$ for $1\leq s\leq n-2.$   Since $ x_1\cdots x_n $ is the only minimal generator of degree $ n $ in $ I^{(n-1)} ,$ there is no minimal generator of degree $ n $ in $ {\tilde{ I}}^{(n-1)}$ as  $ \deg( x_1x_2\cdots {x_{n}}^{w_n} ) \geq n+1. $   By the definition of $ I^{(s)} ,$   $ x_1^2x_2^2\cdots {x_{n-1}} \in I^{(n-1)}.$ Then by   Theorem \ref{sym.theorem.1},  $\Phi(x_1^2x_2^2\cdots {x_{n-1}}) = x_1^2x_2^2\cdots {x_{n-1}} \in  \tilde{ I}^{(n-1)} .$ Since $ \deg(x_1^2x_2^2\cdots {x_{n-1}}) = n+1,$  $ x_1^2x_2^2\cdots {x_{n-1}}$ is one of the least degree  generator of $ {\tilde{ I}}^{(n-1)}  $   and hence $ \alpha({\tilde{ I}}^{(n-1)}) = n+1.$ Here 
		$ x_1^2x_2^2\cdots {x_{n-1}}  = (x_1x_2\cdots {x_{n-1}})(x_1x_2) \in \tilde{ I}^{(n-1)} $ where 	$  x_1x_2\cdots {x_{n-1}}$ and $x_1x_2$ are one of the least degree generators of  ${\tilde{ I}}^{(n-2)}$   and        ${\tilde{ I}}$ respectively. Therefore one of  
	the least degree generator of   $ {\tilde{ I}}^{(n-1)}  $ can be generated by the least degree generators of ${\tilde{ I}}^{(n-2)}$   and      ${\tilde{ I}}.$
	\vspace*{0.2cm}\\
		By the above argument, one of the least degree  generators of   $ {\tilde{ I}}^{(s)}  $ can be generated by the least degree  generators of ${\tilde{ I}},$ ${\tilde{ I}^{(2)}}, \ldots,   {\tilde{ I}}^{(n-2)}.$    Now from (1), by assuming $ r_{n-1} = 0, $ it follows that for $s\geq n,$ 
		$\alpha({\tilde{ I}}^{(s)})=\min\{2r_1+3r_2+\dots+(n-1)r_{n-2}\mid s=r_1+2r_2+\dots+(n-2)r_{n-2}  \}.$   Now $2r_1+3r_2+\dots+(n-1)r_{n-2}=s+r_1+r_2+\dots+r_{n-2}.$ Then it is equivalent to find the minimum of ${r=r_1+r_2\dots+r_{n-2}}$ with the condition $s=r_1+2r_2+\dots+(n-2)r_{n-2}.$ Write $s=k(n-2)+p$ for some $k\in \mathbb{Z}$ and $0\leq p\leq n-3.$ Here the minimum value of $r$ will occur for maximum value of $r_{n-2}$ and the maximum value of $r_{n-2}$ is $k.$ So the minimal generating degree term will come from ${\tilde{ I}}^{(n-2)^k}{\tilde{ I}}^{(p)}.$ Thus
		\begin{align*}
		\alpha({\tilde{ I}}^{(s)})  =  \left\{\begin{array}{ll}
		k(n-1)+(p+1)& \text{ if $p \neq 0,$} \\
		k(n-1) & \text{  if $p=0$ }\\
		\end{array} \right.
		\end{align*}
		\begin{align*}
		\alpha({\tilde{ I}}^{(s)})=\left\{\begin{array}{ll}
		s+1+ \Bigl\lfloor\dfrac{s}{n-2}\Bigr\rfloor\qquad & \text{ if $p \neq 0,$} \\
		& \\
		s+ \Bigl\lfloor\dfrac{s}{n-2}\Bigr\rfloor\qquad & \text{  if $p=0$ }\\
		\end{array} \right.
		\end{align*}  
		Therefore   $\displaystyle{ \alpha{({\tilde{ I}}^{(s)})}
		=  s+ \Bigl\lceil\dfrac{s}{n-2}\Bigr\rceil}.$
		Also $\displaystyle{\frac{s}{n-2}\leq\Bigl\lceil\dfrac{s}{n-2}\Bigr\rceil\leq \frac{s}{n-2}+1}.$ Hence \\   $$\displaystyle{\widehat{\alpha}{({\tilde{ I}})}=\lim_{s\rightarrow\infty}{} \frac{\alpha{({\tilde{ I}}^{(s)})}}{s} = 1+ \frac{1}{n-2}=\frac{n-1}{n-2}}. $$
	\end{enumerate}
\end{proof}

\begin{remark}                     
	In Corollaries \ref{unicycle} and  \ref{tournament}, we proved the equality of ordinary  and symbolic powers for weighted  oriented unicyclic graphs and complete graphs under certain orientations and weights. If   we change the orientation and weights, we may lose the equality of ordinary and symbolic powers as it was shown in  Propositions \ref{uni-sym} and \ref{com-sym}. 
		
\begin{figure}[!ht] 
		\begin{tikzpicture}
		\begin{scope}[ thick, every node/.style={sloped,allow upside down}] 
		\definecolor{ultramarine}{rgb}{0.07, 0.04, 0.56} 
		\definecolor{zaffre}{rgb}{0.0, 0.08, 0.66}

		\node at (2.8,-0.3) {$x_1$};
		\node at (4.8,-0.3) {$x_5$};
		\node at (2.5,1.7) {$x_{2}$};
		\node at (5.1,1.7) {$x_4$};
		\node at (3.9,2.7) {$x_3$};	
		\node at (2.6,0.1) {$2$};      
		\node at (5,0.1) {$2$};
		\node at (2.1,1.3) {$2$};
		\node at (5.5,1.3) {$2$};
		\node at (3.5,2.5) {$2$};

		\node at (10.4,0.1) {$1$};
		\node at (7.7,1.3) {$2$}; 
		\node at (10.9,1.3) {$2$};
		\node at (8.3,-0.3) {$x_1$};
		\node at (10.4,-0.3) {$x_5$};
		\node at (8,1.7) {$x_{2}$};
		\node at (10.5,1.7) {$x_4$};
		\node at (9.3,2.7) {$x_3$};

		\draw[fill][-,thick] (2.5,1.3) --(3.8,2.4);   
		\draw[fill](8.5,0) --node {\midarrow}(10.1,0); 
		\draw[fill] (10.1,0) --node {\midarrow}(10.6,1.3);
		\draw[fill] (8.5,0) --node {\midarrow}(8,1.3);
		\draw[fill] (9.3,2.4) --node {\midarrow}(10.6,1.3);
		\draw[fill](9.3,2.4)  --node {\midarrow}(8,1.3);   
		\draw [fill] [fill] (8.5,0) circle [radius=0.04];
		\draw [fill] [fill] (10.1,0) circle [radius=0.04];  
		\draw [fill] [fill] (10.6,1.3) circle [radius=0.04];
		\draw [fill] [fill] (9.3,2.4) circle [radius=0.04];
		\draw [fill] [fill] (8,1.3) circle [radius=0.04];
		
		\draw[fill] (4.6,0) --node {\midarrow}(3,0);
		\draw[fill]  (5.1,1.3)  --node {\midarrow}(4.6,0);
		\draw[fill] (3,0) --node {\midarrow}(2.5,1.3);
		\draw[fill] (3.8,2.4) --node {\midarrow}(5.1,1.3);
		\draw[fill] (2.5,1.3)--node {\midarrow}(3.8,2.4) ; 
		\draw [fill] [fill] (3,0) circle [radius=0.04];
		\draw [fill] [fill] (4.6,0) circle [radius=0.04];
		\draw [fill] [fill] (5.1,1.3) circle [radius=0.04];
		\draw [fill] [fill] (3.8,2.4) circle [radius=0.04];
		\draw [fill] [fill] (2.5,1.3) circle [radius=0.04];

		\node at (3.8,-0.5) {$D_1$};	
		\node at (9.3,-0.5) {$D_2$};

		\end{scope}
		\end{tikzpicture} 	\caption{Two weighted  oriented odd cycle's of length $5$.}\label{oddcycle.figure}  
	\end{figure}
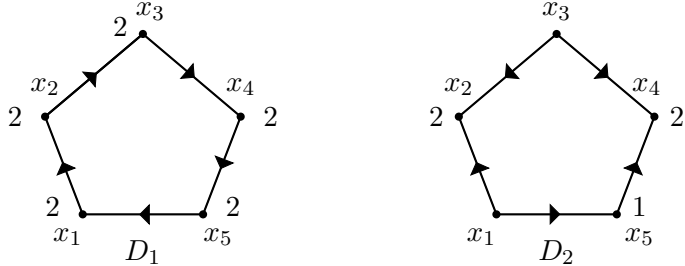
	For example consider $D_1$ and $D_2$   be the weighted oriented odd cycles as in Figure \ref{oddcycle.figure}.
Then $I(D_1) = (x_1x_2^{2},x_2x_3^{2},x_3x_4^{2},x_4x_5^{2},x_5x_1^2)$ 
and $I(D_2) = (x_1x_2^2,x_2^2x_3,x_3x_4^2,x_4^2x_5,x_1x_5).$ In $D_1,$ edges are naturally oriented and $w(x_i) = 2$ for each vertex $x_i$. In $D_2,$ vertices of $V^+(D_2)= \{x_2,x_4  \}$ are sinks.	By Corollary \ref{unicycle}, $ I(D_1)^{(s)} = I(D_1)^s $ for all $s \geq 2$ where as  by Proposition \ref{uni-sym}, $ I(D_2)^{(s)} \neq I(D_2)^s $ for all $s \geq 3$. 
\end{remark}

\begin{proposition}\label{clique-sym}   
Let $\tilde{I}$ be the edge ideal of a  weighted oriented  graph $ D $ where the elements of $ V^{+} $ are sinks and  let $ G $ be the underlying  graph  of $ D $
which  is a clique sum of two odd cycles   $C_{2n+1}=(x_1,\ldots,x_{2n+1})$ and $C_{2n+1}^{\prime} = (x_1,y_2,\ldots,y_{2n+1})$ of same length with a common vertex $x_1$.
 Let $ c = (\prod v| v \in  V(C_{2n+1}) \setminus V^{+})( \prod {v^{w(v)}}| v \in  V(C_{2n+1}) \cap V^{+}) $ and    $ c^{\prime} = (\prod v| v \in  V(C^\prime_{2n+1}) \setminus V^{+} )( \prod {v^{w(v)}}| v \in  V(C_{2n+1}^{\prime})\cap V^{+})   .$    
   
 \begin{enumerate}
	 \item Then  ${\tilde{ I}}^{(s)}={\tilde{ I}}^s$ for $1\leq s\leq n$
	and ${\tilde{ I}}^{(n+1)}= {\tilde{ I}}^{n+1} + (c) + (c^\prime).$

	\item Let $s\in \mathbb{N}$ and write $s=k(n+1)+r$ for some $k \in\mathbb{Z} $ and $0\leq r\leq n.$ Then  $${\tilde{ I}}^{(s)}=\displaystyle  {\sum_{p+q=t=0}^{k}{{\tilde{ I}}^{s-t(n+1)}(c)^p(c^\prime)^q}}.$$ 
	\item If only one of the two cycles does not contain any  element of $ V^{+} $, then for    $ s\in \mathbb{N}$, $ \alpha{({\tilde{ I}}^{(s)})}$
	$\displaystyle{= 2s- \Bigl\lfloor\dfrac{s}{n+1}\Bigr\rfloor }.$
	In particular,
	 the Waldschmidt constant of ${\tilde{ I}}$ is given by \begin{align*}\widehat{\alpha}{({\tilde{ I}})}=\frac{2n+1}{n+1}.\end{align*} 
		\item If  both  cycles    contain some element of $ V^{+} $, then for    $ s\in \mathbb{N}$, $ \alpha{({\tilde{ I}}^{(s)})}$
	$ =  2s .$
In particular, the Waldschmidt constant of ${\tilde{ I}}$ is given by $ \widehat{\alpha}{({\tilde{ I}})}=2. $

\end{enumerate} 	
\end{proposition}

\begin{proof}
Let $I$ be the edge ideal of $ G $ and $ \Phi $ is same as defined in  Notation \ref{phi}.

\begin{enumerate}
	\item  Since $ G $ is the clique sum of two odd cycles    of same length with a common vertex, by  \cite[Corollary 3.15]{bidwan}, we have   $ {{I}^{(k)}} = {{I}^{k}} $ for $  1\leq k \leq n $ 
	and $ {{I}^{(n+1)}} = {{I}^{n+1}} +(x_1\cdots x_{2n+1}) + (x_1y_2\cdots y_{2n+1}). $ Thus by Theorem \ref{sym.theorem.1}, $\Phi( {{I}^{(k)}})  =  \Phi({{I}^{k}}) ,$ i.e., $ {\tilde{I}^{(k)}} = {\tilde{I}^{k}} $  for $  1\leq k \leq n $  and  $ \Phi({{I}^{(n+1)}}) = \Phi({{I}^{n+1}}) + \Phi(x_1\cdots x_{2n+1}) + \Phi(x_1y_2\cdots y_{2n+1}) ,$ i.e., $ {\tilde{I}^{(n+1)}} = {\tilde{I}^{n+1}} +   (c) + (c^{\prime})      . $  	
	
	\item    Let $s\in \mathbb{N}$ and write $s=k(n+1)+r$ for some $k \in\mathbb{Z} $ and $0\leq r\leq n.$ By \cite[Theorem 3.16]{bidwan}, we have $I^{(s)}=\displaystyle  {\sum_{p+q=t=0}^{k}{I^{s-t(n+1)}(x_1\cdots x_{2n+1})^p(x_1y_2\cdots y_{2n+1})^q}}.$
	   Thus  by Theorem \ref{sym.theorem.1}, applying $ \Phi $ on both sides we get  ${\tilde{ I}}^{(s)}=\displaystyle  {\sum_{p+q=t=0}^{k}{{\tilde{ I}}^{s-t(n+1)}(c)^p(c^\prime)^q}}.$

	\item Assume only one of the both cycles does not contain any  element   of $ V^{+} .$ Without loss of generality let  $ C_{2n+1} $ be the one without any element of $ V^{+} .$   Thus by (1),    $ {\tilde{I}^{(n+1)}} = {\tilde{I}^{n+1}} +  (c) + (c^{\prime}) =  {\tilde{I}^{n+1}} +     (x_1\cdots x_{2n+1}) + (c^{\prime}). $ Since all the edges of  cycle $ C_{2n+1} $ are without any element
	 of $ V^{+} ,$   $ \alpha(\tilde{I}^{s})=\alpha (I^{s}) = 2s $ for all $ s. $      Here $ \alpha(c) =2n+1  $ and  $ \alpha(c^\prime)  \geq 2n+2  .$ Thus we have  $ \alpha(\tilde{I}^{(n+1)})=\alpha (I^{(n+1)}) = 2n+1 .$ Since ${I}^{(s)}  $ is generated by $ {I}  $ and $ {I}^{(n+1)}, $  $\tilde{I}^{(s)}  $ is generated by $ \tilde{I}  $ and $\tilde{I}^{(n+1)} $ which implies $ \alpha{(\tilde{I}^{(s)})}= \alpha{({I}^{(s)})} .$ Therefore by   \cite[Theorem 3.12]{bidwan}, we have $\displaystyle {\alpha{(\tilde{I}^{(s)})}
	 =  2s-\Bigl\lfloor\dfrac{s}{n+1}\Bigr\rfloor }. $   Hence $\displaystyle{\widehat{\alpha}{(\tilde{I})} = 2 - \frac{1}{n+1}=\frac{2n+1}{n+1}}.$    
	\item   Assume   both  cycles    contain some element of $ V^{+} .$ By (1), we have $ {\tilde{ I}}^{(s)} = {\tilde{ I}}^{s} +  \displaystyle  {\sum_{p+q=t=1}^{k}{{\tilde{ I}}^{s-t(n+1)}(c)^p(c^\prime)^q}}.$ Since $ G $ contains   odd cycles, there   exists   an edge in each odd cycle without any element of $ V^{+} .$ Thus  $ \alpha(\tilde{I}^s)=\alpha (I^s)= 2s $ for all $ s. $ Since each of the both  cycles    contain some element of $ V^{+} ,$ $   \alpha(c) \geq 2(n+1) ,$  $   \alpha(c^{\prime}) \geq 2(n+1) ,$ $   \alpha(c^p) \geq 2p(n+1) $ and $   \alpha({c^{\prime}}^q) \geq 2q(n+1) .$ Also $ \alpha (\tilde{I}^{s-t(n+1)}) = 2s - 2t(n+1) $ and $ \alpha ({\tilde{I}}^{s-t(n+1)}(c)^p(c^\prime)^q) \geq 2s - 2t(n+1) + 2p (n+1) + 2q (n+1)  = 2s  $ for  $ 1 \leq t \leq k. $ Since $ \alpha(\tilde{I}^s) = 2s,   $   $\alpha(\tilde{I}^{(s)}) = \alpha(\tilde{I}^s) = 2s.   $  Hence   $\widehat{\alpha}{(\tilde{I})}=   2.$     	  
\end{enumerate}    \end{proof}
\section{Regularity in weighted oriented odd cycle}
In this section, we focus on  weighted oriented odd cycles where the elements of $V^+$ are sinks and we set $ V^+ \neq \phi. $ We  first show that $\reg \tilde{I}^{(s)}  \leq \reg \tilde{I}^s$ and then we explicitly compute $\reg ( \tilde{I}^{(s)} /\tilde{I}^{s})$. As a Corollary, we conclude that  the equality  $\reg \tilde{I}^{(s)}  = \reg \tilde{I}^s$  occurs if  $V^+$ has   only one  element.     
\begin{lemma} \label{maximal.3}    
	Let $\tilde{I}$ be the edge ideal of a  weighted oriented cycle $ D $ where the elements of $ V^{+} $ are sinks and let $ G $ be the underlying graph  of $ D $ which is the cycle $ C_{2n+1} = (x_1, \ldots , x_{2n+1} ).$   Let $ w_i = w(x_i)  $ where $ x_i \in V^+ ,$  $ w_v := \max\{ w_i ~|~x_i \in V^+  \}  $ and $\displaystyle{ w = \sum_{x_i \in V^+  }(w_i - 1)        .      }$ For $ s \in \mathbb{N} ,$ let  $ s = k(n + 1) + r $ for some $ k \in  \mathbb{Z}  $ and
	$ 0 \leq r \leq n. $  Then   $\m^{(t-1)w_v + w } (c)^t \nsubseteq   \tilde{I}^{t(n+1)}   $ for   $ 1 \leq t \leq k $  where $ c = (\prod x_j| x_j \notin V^{+})( \prod {x_j^{w_j}}| x_j \in V^{+})  $  and $ \m = (x_1, \ldots , x_{2n+1}) $ is the maximal homogeneous ideal.    
	
\end{lemma}

\begin{proof} 
	Here   $ c = (\prod x_j| x_j \notin V^{+})( \prod {x_j^{w_j}}| x_j \in V^{+})   $  and $ x_v $ be an  element of $ V^{+}  $   with maximum weight. Here $$\tilde{I}  =    (x_ix_j^{w_j}~ \mbox{if} ~ x_j \in V^{+} ~ \mbox{or} ~ x_ix_j ~ \mbox{if} ~ x_i, x_j \notin V^{+}~|~\{x_i,x_j\} \in E(C_{2n+1})) .$$     
	Let  $ \displaystyle{ f = \prod_{ 
			 x_j \in V^{+}} {x_j^{(w_j  - 1)}}}.$ Then $ \deg(f) = w  $ and  $ f  $ is a minimal generator of $\m^w .$  Let  $ g    = {x_v}^{(t-1)w_v } f c^{t}$ for some $t$ where   $ 1 \leq t \leq k .$ Then  $  g   \in    \m^{(t-1)w_v }  \m^w(c)^t = \m^{(t-1)w_v + w } (c)^t   . $  We want to prove that  $ g     \notin  \tilde{I}^{t(n+1)}    .$

	  Note that  $ c $ can be expressed as the product of $ n $ minimal generators of $ \tilde{I} $ and any $ x_i $ where $x_i \notin V^+ $ or  the product of $ n $ minimal generators of $ \tilde{I} $ and any $ x_i^{w_i}$ where $x_i \in V^+ .$ So the product of $ c $ only with any $ x_j^{w_j} $ for some $ x_j \in V^+ $  or   $ x_j $ for some $ x_j \notin V^+ $ is exactly the product  of  $ n+1 $ minimal generators of $ \tilde{I}  .$  Hence the product of $ c^t $ only with $t$ number of $ x_j^{w_j} $ for some $ x_j \in V^+ $  or   $ x_j $ for some $ x_j \notin V^+ $ is exactly the  product  of  $ t(n+1) $ minimal generators of $ \tilde{I}  .$ Since $ \displaystyle{ f = \prod_{ 
	  		x_j \in V^{+}} {x_j^{(w_j  - 1)}}}$  is neither   a multiple of $ x_j^{w_j} $ for some $ x_j \in V^+ $  nor   $ x_j $ for some $ x_j \notin V^+ ,$   $  {x_v}^{(t-1)w_v } f $ can not be a multiple of product of $t$ number of $ x_j^{w_j} $ for some $ x_j \in V^+ $  or   $ x_j $ for some $ x_j \notin V^+ $. Note that  $  {x_v}^{(t-1)w_v } f $ is a multiple of product of only $t-1$ number of $ x_j^{w_j} $ for  $ x_j=x_v \in V^+ .$ Therefore  ${x_v}^{(t-1)w_v } f c^{t}$ is not a multiple of product  of any  $ t(n+1) $ minimal generators of $ \tilde{I}  .$ So $ g  = {x_v}^{(t-1)w_v } f c^{t}  \notin  \tilde{I}^{t(n+1)}    .$ Hence    $\m^{(t-1)w_v + w } (c)^t \nsubseteq   \tilde{I}^{t(n+1)}   $ for   $ 1 \leq t \leq k $.   
\end{proof} 

\begin{lemma} \label{maximal.1}   
	Let $\tilde{I}$ be the edge ideal of a  weighted oriented cycle $ D $ where the elements of $ V^{+} $ are sinks and let $ G $ be the underlying graph  of $ D $ which is the cycle $ C_{2n+1} = (x_1, \ldots , x_{2n+1} ).$   Let $ w_i = w(x_i)  $ where $ x_i \in V^+ ,$  $ w_v := \max\{ w_i ~|~x_i \in V^+  \}  $ and $\displaystyle{ w = \sum_{x_i \in V^+  }(w_i - 1)        .      }$ For $ s \in \mathbb{N} ,$ let  $ s = k(n + 1) + r $ for some $ k \in  \mathbb{Z}  $ and
$ 0 \leq r \leq n. $  Then   $\m^{(k-1)w_v + w + 1} \tilde{I}^{(s)} \subseteq  \tilde{I}^{s}  $ where $ \m = (x_1, \ldots , x_{2n+1}) $ is the maximal homogeneous ideal.        
	
\end{lemma}

\begin{proof} 
	Let   $ c = (\prod x_j| x_j \notin V^{+})( \prod {x_j^{w_j}}| x_j \in V^{+})   .$  Let  $ x_v $ be a vertex of $ V^{+} $  with maximum weight. Here $$\tilde{I}  =    (x_ix_j^{w_j}~ \mbox{if} ~ x_j \in V^{+} ~ \mbox{or} ~ x_ix_j ~ \mbox{if} ~ x_i, x_j \notin V^{+}~|~\{x_i,x_j\} \in E(C_{2n+1})) .$$ 
By Proposition \ref{uni-sym}, we have $ \tilde{I}^{(s)} =  \tilde{I}^s + \displaystyle{ \sum_{t=1}^{k} \tilde{I}^{s-t(n+1)}  (c)^t  .}$ Since $\displaystyle{ w+1 = \sum_{x_i \in V^+  }(w_i - 1)}+1,$ any   element  of $ \m^{w+1} $ can be viewed as a multiple of  $ x_j^{w_j} $ for some $ x_j \in V^+ $  or a multiple of $ x_j $ for some $ x_j \notin V^+ .$ Note that the product of $ c $ only with any $ x_j^{w_j} $ for some $ x_j \in V^+ $  or   $ x_j $ for some $ x_j \notin V^+ $ is exactly the  product  of  $ n+1 $ minimal generators of $ \tilde{I}  .$ Thus any element of $ \m^{w+1} (c)  $ can be represented as a multiple of product of   $ n+1 $ minimal generators of $ \tilde{I}  $ and hence $ \m^{w+1} (c) \subseteq \tilde{I}^{n+1} .     $ First we want to show that $    \m^{w_v}  \m^ {w + 1} (c)    \subseteq \m^{w+1}\tilde{I}^{n+1}.$
\vspace*{0.2cm}\\   
	Let $ t_1 \in     \m^{w_v}  \m^ {w + 1} (c)  $ such that  $ t_1 = t_2 t_3 $ where $ t_2 \in \m^{w_v}  $ and $ t_3 \in   \m^ {w + 1} (c) . $ Since $  \m^ {w + 1} (c) \subseteq \tilde{I}^{n+1}, $ $ t_3 $ is  a   multiple of   product  of $ n+1 $ minimal generators of $ \tilde{I}  .$ 
	Let the product of those $ n+1 $ minimal generators of $ \tilde{I}  $ is $ t_4 ,$ i.e., $ t_4 \in  \tilde{I}^{n+1}.$  Then $ t_5 = t_3/t_4 $ is a monomial of degree at least $\displaystyle{ \sum_{x_i \in V^+ \setminus \{x_v\} }(w_i - 1)  }   .$
	Here $ t_2 t_5 \in \m^{w+1} $ because  $ w_v + \displaystyle{ \sum_{x_i \in V^+ \setminus \{x_v\} }(w_i - 1)  }  =   w+1  .$ So $ t_1 = t_2 t_3 = t_2 t_4 t_5 = (t_2 t_5) t_4  \in \m^{w+1}\tilde{I}^{n+1}.$   \mbox{Thus} $    \m^{w_v}  \m^ {w + 1} (c)    \subseteq \m^{w+1}\tilde{I}^{n+1}.$   Hence      
\begin{align*}
  \m^{(t-1)w_v + w + 1} (c)^{t} &=     \m^{(t-1)w_v} (c)^{t-1} \m^ {w + 1} (c)\\  
  &= \m^{(t-2)w_v} (c)^{t-1} \m^{w_v}   \m^ {w + 1}  (c)\\
&\subseteq \m^{(t-2)w_v} (c)^{t-1}   \m^ {w + 1}  \tilde{I}^{n+1}\\      
&= \m^{(t-3)w_v} (c)^{t-2} \m^{w_v}  \m^ {w + 1}  (c) \tilde{I}^{n+1}\\ 
&\subseteq \m^{(t-3)w_v} (c)^{t-2}   \m^ {w + 1}  \tilde{I}^{2(n+1)}\\
&~~~~\vdots\\
& \subseteq     \m^ {w + 1} (c) \tilde{I}^{(t-1)(n+1)}\\
&\subseteq \tilde{I}^{n+1} \tilde{I}^{(t-1)(n+1)}\\&=\tilde{I}^{t(n+1)}.  
\end{align*}     
Therefore
 \begin{align*}
   \m^{(k-1)w_v + w + 1} \tilde{I}^{s-t(n+1)} (c)^t  &= \m^{(k-t)w_v} \tilde{I}^{s-t(n+1)} \m^{(t-1)w_v + w + 1}   (c)^t \\   
     &\subseteq      \m^{(k-t)w_v} \tilde{I}^{s-t(n+1)} \tilde{I}^{t(n+1)}\\ 
     &\subseteq  \m^{(k-t)w_v} \tilde{I}^{s}\\
      &\subseteq \tilde{I}^{s}  ~    \mbox{for}    ~  1 \leq t \leq k .  
 \end{align*}
  Hence $\m^{(k-1)w_v + w + 1} \tilde{I}^{(s)} = \m^{(k-1)w_v + w + 1}(    \tilde{I}^s + \displaystyle{ \sum_{t=1}^{k} \tilde{I}^{s-t(n+1)}  (c)^t    } )  \subseteq   \tilde{I}^{s}.   $ 
\end{proof}

\begin{theorem}\label{maximal.2} 
	Let $\tilde{I}$ be the edge ideal of a  weighted oriented cycle $ D $ where the elements of $ V^{+} $ are sinks and let $ G $ be the underlying graph  of $ D $ which is the cycle $ C_{2n+1} = (x_1, \ldots , x_{2n+1} ).$   Let $ w_i = w(x_i)  $ where $ x_i \in V^+ ,$  $ w_v := \max\{ w_i ~|~x_i \in V^+  \}  $ and $\displaystyle{ w = \sum_{x_i \in V^+  }(w_i - 1).}$ For $ s \in \mathbb{N} ,$ let  $ s = k(n + 1) + r $ for some $ k \in  \mathbb{Z}  $ and
$ 0 \leq r \leq n. $   Then     $\reg \tilde{I}^{(s)}  \leq \reg \tilde{I}^s.$
\end{theorem}
\begin{proof}
Let    $ \m = (x_1, \ldots , x_{2n+1}) $ is the maximal homogeneous ideal.	By Lemma \ref{maximal.1}, we have $\m^{(k-1)w_v + w + 1}{\tilde{I}}^{(s)}\subseteq \tilde{I}^s.$ Thus $\tilde{I}^{(s)}/\tilde{I}^s$ is an Artinian module. Therefore $\dim (\tilde{I}^{(s)}/\tilde{I}^s)=0$ and hence $H_{\m }^i(\tilde{I}^{(s)}/\tilde{I}^s)=0$ for $i>0.$ Consider the following short exact sequence
	$$0\rightarrow \tilde{I}^{(s)}/\tilde{I}^s \rightarrow R/\tilde{I}^s \rightarrow R/\tilde{I}^{(s)}\rightarrow 0. $$
	Applying local cohomology functor we get $H_{\m}^i(R/\tilde{I}^{(s)})\cong H_{\m}^i(R/\tilde{I}^{s})$ for $i\geq 1$ and the following short exact sequence
	\begin{equation} \label{exact}
	0\rightarrow H_{\m}^0(\tilde{I}^{(s)}/\tilde{I}^s) \rightarrow H_{\m}^0(R/\tilde{I}^s) \rightarrow H_{\m}^0(R/\tilde{I}^{(s)})\rightarrow 0.
	\end{equation}
	
	Now by the exact sequence $(\ref{exact}),$  we have $a_0(R/\tilde{I}^{(s)})\leq a_0(R/\tilde{I}^s).$
	Thus we can conclude that $\reg (R/\tilde{I}^{(s)})=\max\{a_i(R/\tilde{I}^{(s)})+i~|~i\geq 0\}\leq\max\{a_i(R/\tilde{I}^{s})+i~|~i\geq 0\}=\reg(R/\tilde{I}^s). $ Hence $\reg \tilde{I}^{(s)}\leq \reg \tilde{I}^s.$
\end{proof}

\begin{theorem} \label{maximal.4}  
		Let $\tilde{I}$ be the edge ideal of a  weighted oriented cycle $ D $ where the elements  of $ V^{+} $ are sinks and let $ G $ be the underlying graph  of $ D $ which is the cycle $ C_{2n+1} = (x_1, \ldots , x_{2n+1} ).$   Let $ w_i = w(x_i)  $ where $ x_i \in V^+ ,$  $ w_v := \max\{ w_i ~|~x_i \in V^+  \}  $ and $\displaystyle{ w = \sum_{x_i \in V^+  }(w_i - 1)        .      }$ For $ s \in \mathbb{N} ,$ let  $ s = k(n + 1) + r $ for some $ k \in  \mathbb{Z}  $ and
	$ 0 \leq r \leq n. $    Then   $\reg ( \tilde{I}^{(s)} /\tilde{I}^{s}) = (s-n-1)(1+w_v) + 2w + 2n  + 1      $ for $ s \geq n+1. $    
\end{theorem}

\begin{proof}
We fix $s \geq  n+1.$		Let   $ c = (\prod x_j| x_j \notin V^{+})( \prod {x_j^{w_j}}| x_j \in V^{+})   $ and   $ \m = (x_1, \ldots , x_{2n+1}) $ be the maximal homogeneous ideal. Let $ x_v $ be an  element of $ V^{+} $  with maximum weight.   
	By Lemma \ref{maximal.1}, we have $\m^{(k-1)w_v + w + 1}\tilde{I}^{(s)}\subseteq \tilde{I}^s.$ Thus $\tilde{I}^{(s)}/\tilde{I}^s$ is an Artinian module. Therefore $\dim (\tilde{I}^{(s)}/\tilde{I}^s)=0,$     $H_{\m }^0(\tilde{I}^{(s)}/\tilde{I}^s)=\tilde{I}^{(s)}/\tilde{I}^s$ and  $H_{\m }^i(\tilde{I}^{(s)}/\tilde{I}^s)=0$ for $i>0.$ So $\reg (\tilde{I}^{(s)}/\tilde{I}^s) =   a_0(\tilde{I}^{(s)} /\tilde{I}^{s}).$ We want to prove that one of the	maximum degree  element of $I^{(s)}$ which  does not lie in   $ \tilde{I}^{s} $ will   come from the ideal      $ \tilde{I}^{s-(n+1)}(c)^1 $.
\vspace*{0.3cm}\\
 By Proposition \ref{uni-sym}, we have $ \tilde{I}^{(s)} =  \tilde{I}^s + \displaystyle{ \sum_{t=1}^{k} \tilde{I}^{s-t(n+1)}  (c)^t  .}$ Here $\deg(c)= 2n + 1 + w  $ and $\deg(c^t)= t(2n + 1 + w)  .$ Without loss of generality we can assume that $x_v\neq x_{2n+1}$. Then $x_v^{w_v}x_{v+1}$ is a maximum degree minimal generator of $\tilde{I}$ which implies that one maximum degree minimal generator of $\tilde{I}^{s-t(n+1)}$ is $(x_v^{w_v}x_{v+1})^{s-t(n+1)}$ and its degree is $(s-t(n+1))(w_v+1).$    
  From the proof of Lemma \ref{maximal.1}, we observe that
	$\m^{(t-1)w_v + w + 1 } (c)^t \subseteq   \tilde{I}^{t(n+1)}   $ for   $ 1 \leq t \leq k  .$
	Thus by      Lemma \ref{maximal.3}, the maximum value of $ u $ such that  $\m^u  (c)^t \nsubseteq \tilde{I}^{t(n+1)}  $ is $ (t-1)w_v + w  $   for   $ 1 \leq t \leq k. $   So the maximum degree of an element  of  the ideal  $ (c)^t $ which  does not lie in $ \tilde{I}^{t(n+1)} $ is $  (t-1)w_v + w + t(2n + 1 + w)         $  for   $ 1 \leq t \leq k  .$ Hence the             maximum degree of an element of   
$ \tilde{I}^{s-t(n+1)} (c)^t  $ which  does not lie   in $\tilde{I}^{s-t(n+1)} \tilde{I}^{t(n+1)} = \tilde{I}^{s} $ is $  (s-t(n+1))(w_v+1)    +  (t-1)w_v + w + t(2n+1+w)         $  for   $ 1 \leq t \leq k. $ Let $ d(t) = (s-t(n+1))(w_v+1)    +  (t-1)w_v + w + t(2n+1+w)        $ for   $ 1 \leq t \leq k. $
If $ t_1 < t_2 $, then
\begin{align*}
 d(t_1) - d(t_2) &= (s-t_1(n+1))(w_v+1)    +  (t_1-1)w_v + w + t_1(2n+1+w) \\
&\hspace*{0.3cm} -   (s-t_2(n+1))(w_v+1)   -      (t_2-1)w_v     - w - t_2(2n+1+w)\\
&=       (t_2 - t_1)         (n+1)(w_v+1)    -  (t_2-t_1) w_v  - (t_2 - t_1 )(2n + 1 + w)\\
 &= (t_2-t_1)( (n+1)(w_v + 1) - w_v - (2n+1+w)   )\\
 &  = (t_2-t_1)( n(w_v + 1)  - (2n+w)   )\\
 &=   (t_2-t_1)( n(w_v-1)   + 2n - 2n - w  )\\
 &=   (t_2-t_1)( n(w_v-1) - w  )  .                
\end{align*}
Since $|E(D)|=2n+1,$ there are at most $ n $     elements of $ V^+. $ So $ n(w_v - 1) \geq \displaystyle{  \sum_{x_i \in V^+  }(w_i - 1)      } = w $ which implies that $ d(t_1) - d(t_2) \geq 0.  $ Thus $ d(1)   \geq d(2)\geq \cdots \geq d(k)  .$         
Therefore   
one of  the maximum degree  elements   of  $ \tilde{I}^{(s)} $   which  does not lie in   $ \tilde{I}^{s} $ will   come from     $ \tilde{I}^{s-(n+1)}(c)^1 $ for $ t=1 $ and its degree is $d(1) = (s-n-1)  (w_v+1) + w + (2n+1+w) = (s-n-1)  (w_v+1) + 2w + 2n+1.  $  Hence $ \reg(\tilde{I}^{(s)} /\tilde{I}^{s}) =   a_0(\tilde{I}^{(s)} /\tilde{I}^{s}) = (s-n-1)  (w_v+1) + 2w + 2n+1   $   for $s \geq n+1.$              	
\end{proof}

\begin{corollary}\label{maximal.5}
			Let $\tilde{I}$ be the edge ideal of a  weighted oriented cycle $ D $ where the elements of $ V^{+} $ are sinks and let $ G $ be the underlying graph  of $ D $ which is the cycle $ C_{2n+1} = (x_1, \ldots , x_{2n+1} ).$   Let  $ V^+ = \{x_v\}  
			$ and $ w_v = w(x_v) $ for some $ x_v \in V(C_{2n+1}) .$   For $ s \in \mathbb{N} ,$ let  $ s = k(n + 1) + r $ for some $ k \in  \mathbb{Z}  $ and
	$ 0 \leq r \leq n. $      
 Then   $   \reg\tilde{I}^{(s)} =  \reg\tilde{I}^{s}  $ for all $ s. $   
	
\end{corollary}

\begin{proof}
  Let $ I $ be the edge ideal of $G $  and $ \Phi $ is same as defined in  Notation \ref{phi}.	
First we fix $s \geq n+1.$	Consider the following short exact sequence
	\begin{equation} \label{exact.1}
	0\rightarrow \tilde{I}^{(s)}/\tilde{I}^s \rightarrow R/\tilde{I}^s \rightarrow R/\tilde{I}^{(s)}\rightarrow 0. 
	\end{equation}
Without loss of generality we can assume that $ V^+ =  \{x_1\}
.$ Thus by Theorem \ref{maximal.4}, we have    $\reg(\tilde{I}^{(s)}/\tilde{I}^{s}) = (s-n-1) (w_1 + 1) + 2(w_1 - 1) + 2n + 1     .$ Here $x_2x_1 \in I$ and by definition of $I^{(s)},$ we have $(x_2x_1)^{s}\in {I}^{(s)}.$ Observe that $(x_2x_1 )^{s} $ is a minimal generator of  $ {I}^{(s)}  .$   Then by   Theorem \ref{sym.theorem.1},  $\Phi((x_2x_1)^{s}) =(x_2x_1^{w_1} )^{s} \in \tilde{I}^{(s)}$ and $(x_2x_1^{w_1} )^{s} $ is a minimal generator of degree $ s(w_1 + 1)  $ in  $ \tilde{I}^{(s)}  .$    
Since there exist a minimal generator of degree $ s(w_1 + 1) $   in $ \tilde{I}^{(s)}  ,$  we have
\begin{align*}
\reg (\tilde{I}^{(s)}) &\geq
	   s(w_1 + 1)\\
	   & =   (s-n-1) (w_1+1) + (n+1) (w_1-1 + 2) \\
	   & =  (s-n-1) (w_1 + 1) +     (n+1)    (w_1-1) + 2(n+1) \\
	   &>  \reg(\tilde{I}^{(s)}/\tilde{I}^{s})  \\
\end{align*}
  as $ n \geq  1 .$     So   $ \reg (R/\tilde{I}^{(s)}) \geq   \reg(\tilde{I}^{(s)}/\tilde{I}^{s}) .$  
	Thus by the exact sequence (\ref{exact.1}) and \cite[Lemma 1.2]{H.T-2016}, we have   $ \reg (R/\tilde{I}^{s})  \leq   \reg (R/\tilde{I}^{(s)}).    $ 
	Also by Theorem \ref{maximal.2}, we have  $ \reg (R/\tilde{I}^{(s)})  \leq   \reg (R/\tilde{I}^{s}).    $ Therefore    $ \reg (R/\tilde{I}^{(s)})  =     \reg (R/\tilde{I}^{s}) ,$ i.e.,    $   \reg\tilde{I}^{(s)} =  \reg\tilde{I}^{s}  $ for $s \geq n+1.$
	By Proposition \ref{uni-sym},  we have $   \reg\tilde{I}^{(s)} =  \reg\tilde{I}^{s}  $ for  $1 \leq s \leq n.$ Hence   $   \reg\tilde{I}^{(s)} =  \reg\tilde{I}^{s}  $ for all $ s. $
	   
\end{proof}    

\textbf{Acknowledgements}
\vspace*{0.15cm}\\
The first author was supported by SERB (grant No.: EMR/2016/006997), India.
The authors are grateful to the referee for his helpful comments and suggestions.

\end{document}